\documentclass{amsart}

\usepackage[english]{babel}

\usepackage{wasysym}
\usepackage{enumerate}
\usepackage{stmaryrd}
\usepackage{cancel}
\usepackage{calrsfs}
\usepackage{mathrsfs}
\usepackage{rotating}
\usepackage[mathcal]{euscript}
\usepackage{hyperref}

\usepackage{amsfonts}
\usepackage{amsthm}
\usepackage{amssymb}

\input xy
\xyoption{all}

\newtheorem{theorem}{Theorem}[section]
\newtheorem{lemma}{Lemma}[section]
\newtheorem{proposition}{Proposition}[section]

\newtheorem{conjecture}{Conjecture}[section]
\newtheorem{algorithm}{Algorithm}[section]

\theoremstyle{definition}
\newtheorem{definition}{Definition}[section]
\newtheorem{example}{Example}[section]

\theoremstyle{remark}
\newtheorem{remark}{Remark}[section]

\numberwithin{equation}{section}

\renewenvironment{proof}{{\noindent \bfseries Proof.}}{\qed \\}

\begin{document}

\title[Di Francesco--Itzykson--G\"ottsche Conjectures for Node Polynomials]{The Di Francesco--Itzykson--G\"ottsche\\Conjectures for Node Polynomials of $\mathbb{P}^{2}$}

\author{Nikolay Qviller}
\address{Centre of Mathematics for Applications, University of Oslo, P.O. Box 1053 Blindern, NO-0316 Oslo, NORWAY}
\email{nikolayq@cma.uio.no}



\date{\today}

\begin{abstract}
For a smooth, irreducible projective surface $S$ over $\mathbb{C},$ the number of $r$-nodal curves in an ample linear system $|\mathscr{L}|$ (where $\mathscr{L}$ is a line bundle on $S$) can be expressed using the $r$th Bell polynomial $P_{r}$ in universal functions $a_{i}, 1 \leq i \leq r,$ of $(S,\mathscr{L}),$ which are $\mathbb{Z}$-linear polynomials in the four Chern numbers of $S$ and $\mathscr{L}.$ We use this result to establish a proof of the classical shape conjectures of Di Francesco--Itzykson and G\"ottsche governing node polynomials in the case of $\mathbb{P}^{2}.$ We also give a recursive procedure which provides the $\mathscr{L}^{2}$-term of the polynomials $a_{i}.$
\end{abstract}

\maketitle

\tableofcontents

\section{Introduction: Known results and goals}
A well-known problem in algebraic geometry concerns the enumeration of curves in a given linear system on a surface $S,$ subject to specific singularity conditions. The classical example concerns curves of given degree $d$ in the complex projective plane $\mathbb{P}^{2}.$ The simplest singular curves have only nodal (or $A_{1}$) singularities; we refer to them as \textit{nodal curves}. Curves with exactly $r$ nodes and no other singularities are called $r$-nodal.

In the case of $\mathbb{P}^{2},$ the variety of curves of degree $d \geq 1$ corresponds naturally to the projective space $\mathbb{P}^{N}, N := \frac{d(d+3)}{2},$ describing homogeneous polynomials in $x,y,z$ of degree $d.$ For $r \leq {d \choose 2},$ the subset of $r$-nodal curves is an $r$-codimensional subvariety of this $\mathbb{P}^{N}$ (see for instance \cite[\S 1]{Ful}). The degree of this variety is equal to the (finite) number of $r$-nodal curves passing through $N-r$ points in general position in $\mathbb{P}^{2},$ as each such point condition represents a generic hyperplane in the parameter space $\mathbb{P}^{N}$ (cf. \cite[Lemma 4.7]{KP1}). Let $N_{r}(d)$ denote this number. More generally, consider any smooth, projective, (irreducible) surface $S$ equipped with some line bundle $\mathscr{L},$ and let $N_{r}(S,\mathscr{L})$ denote the number of $r$-nodal curves on $S$ which lie in the linear system $|\mathscr{L}|$ and pass through $N-r$ points in general position on $S$ (where $N := \textnormal{dim }|\mathscr{L}|$). Thus, $N_{r}(d)$ corresponds to $N_{r}(\mathbb{P}^{2},\mathscr{O}_{\mathbb{P}^{2}}(d)).$

There has recently been substantial progress in the understanding of these numbers. Consider the four Chern numbers of $(S,\mathscr{L}),$ namely $\partial := \mathscr{L}^{2}, k := \mathscr{L}\mathscr{K}_{S}, s := \mathscr{K}_{S}^{2}$ and $x := c_{2}(S),$ with $\mathscr{K}_{S}$ the canonical sheaf on $S.$ In \cite[Conjecture 2.1]{Got}, G\"ottsche conjectured the existence of universal, rational polynomials $Z_{r}$ in four variables, with $\textnormal{deg }Z_{r} = r$ for all $r \geq 0,$ such that $N_{r}(S,\mathscr{L}) = Z_{r}(\partial,k,s,x)$ provided $\mathscr{L}$ is $(5r-1)$-ample (we will define this notion later). The conjecture was proved by Tzeng \cite[Theorem 1.1]{Tzeng}, and in shorter terms by Kool, Shende and Thomas \cite{KST}. They also refined the result, showing that it is sufficient for $\mathscr{L}$ to be $r$-ample. Since the polynomials $Z_{r}$ enumerate nodal curves, they are called \textit{node polynomials}, in accordance with the terminology used by Kleiman--Piene in \cite{KP1,KP2}. Note that if $S=\mathbb{P}^{2}$ and $\mathscr{L} = \mathscr{O}_{\mathbb{P}^{2}}(d),$ the four Chern numbers become $d^{2},-3d,$ 9 and 3, so the node polynomials become polynomials of degree $2r$ in $d.$

It is convenient to study the generating function of the node polynomials $Z_{r}.$ That is, given $S$ and $\mathscr{L},$ consider the formal power series $\sum_{r \geq 0}Z_{r}(\partial,k,s,x)q^{r}.$ A consequence of Tzeng's proof of \cite[Conjecture 2.1]{Got} is G\"ottsche's \cite[Proposition 2.3]{Got}, the existence of four universal power series $A_{i}(q)$ with rational coefficients such that
\begin{equation}
\label{gen_function}
\sum_{r \geq 0}Z_{r}(\partial,k,s,x)q^{r} = A_{1}(q)^{\partial}A_{2}(q)^{k}A_{3}(q)^{s}A_{4}(q)^{x}.
\end{equation}

Here, we show that this proposition implies (together with some geometric input, cf. Lemma \ref{geom_input}) the existence of universal $\mathbb{Z}$-linear polynomials $a_{i}, i \geq 1,$ in four variables, such that $r!Z_{r}$ is equal to the $r$th \textit{Bell polynomial} (known from combinatorics) in the first $r$ polynomials $a_{i}.$ This clarifies some of the structure of the (huge) node polynomials $Z_{r}.$ Using some power series manipulations we obtain a recursive formula for the coefficient of $\partial$ in $a_{i}(\partial,k,s,x),$ and prove that it is divisible by 3. Also, analysing the Bell polynomial structure further yields a proof of another conjecture of G\"ottsche (see \cite[Remark 4.2]{Got}) on the shape of the node polynomials of $\mathbb{P}^{2}.$ This was a refinement of a somewhat older conjecture of Di Francesco--Itzykson (see the remarks following Proposition 2 in \cite{DI}), and provides information about the general form of the coefficients of the degree $2r$ rational node polynomial.
\begin{conjecture}
\emph{(Di Francesco--Itzykson--G\"ottsche.)}
\label{plane_shape}
The node polynomials of the projective plane, enumerating $r$-nodal curves of degree $d$ in $\mathbb{P}^{2}$ for $d$ sufficiently large, are of the form
\begin{equation}
\frac{3^{r}}{r!}\sum_{\mu=0}^{2r} \frac{1}{\mu!3^{\lfloor \mu/2 \rfloor}}\frac{r!}{(r-\lceil \mu/2 \rceil)!}Q_{\mu}(r)d^{2r-\mu},
\end{equation}
where $Q_{\mu}$ is a polynomial with integer coefficients and degree $\lfloor \mu/2 \rfloor.$
\end{conjecture}
At the moment of writing, the author does not know of any other proof of this statement in full generality, although there have been partial results in this direction. Note that G\"ottsche adds the observation that the only common factors of the coefficients of $Q_{\mu}$ are powers of 2 and 3. We do not prove this.

Historically, the progress towards an understanding of the structure of the node polynomials (both for $\mathbb{P}^{2}$ in particular and in full generality) has been fragmented. For the projective plane, the number $N_{1}(d)$ was known early in the 19th century. Indeed, one observes that the variety of curves having degree $d$ and passing through $d(d+3)/2 - 1$ general points is a pencil, more precisely, $\{Z(\lambda f + \mu g), [\lambda:\mu] \in \mathbb{P}^{1}\},$ where $f$ and $g$ are degree $d$ homogeneous polynomials in $x,y,z$, the zero sets of which both contain the $d(d+3)/2-1$ specified points. The nodal curves in this pencil occur for pairs $(\lambda,\mu)$ such that
\begin{equation}
\lambda \frac{\partial f}{\partial x} + \mu \frac{\partial g}{\partial x} = \lambda \frac{\partial f}{\partial y} + \mu \frac{\partial g}{\partial y} = \lambda \frac{\partial f}{\partial z} + \mu \frac{\partial g}{\partial z} = 0
\end{equation}
(by homogeneity the vanishing of the three derivatives of $\lambda f + \mu g$ implies the vanishing of $\lambda f + \mu g$ itself). This is a system of three equations in two unknowns, $\lambda$ and $\mu,$ so we have solutions provided that
\begin{equation}
\left| \begin{array}{cc} \frac{\partial f}{\partial x} & \frac{\partial g}{\partial x} \\ \frac{\partial f}{\partial z} & \frac{\partial g}{\partial z} \end{array} \right| = \left| \begin{array}{cc} \frac{\partial f}{\partial y} & \frac{\partial g}{\partial y} \\ \frac{\partial f}{\partial z} & \frac{\partial g}{\partial z} \end{array} \right| = 0
\end{equation}
The two curves of degree $2(d-1)$ which are the zero sets of these determinants intersect, by B\'ezout's theorem, in $\left(2(d-1)\right)^{2}$ points. Remove the $(d-1)^{2}$ points where $\frac{\partial f}{\partial z} = \frac{\partial g}{\partial z} = 0,$ then we are left with $N_{1}(d) = 3(d-1)^{2}.$

A similar formula for 2-nodal curves of ``large enough'' degree, $N_{2}(d) = \frac{3}{2}(d-1)(d-2)(3d^{2}-3d-11),$ was given by Cayley in \cite[Art. 33, p. 306]{Cay} in 1863. It was also found by Salmon in \cite[Appendix IV, p. 506]{Sal}. His method consisted of first counting the number of curves which either had two nodes \textit{or} one cusp, then subtracting the number of cuspidal curves. A similar method is used by Roberts in \cite[p. 111--12]{Rob} in 1875, for 3-nodal curves: He gives a polynomial for curves which either have three nodes, or one node and one cusp, or one tacnode, a polynomial for curves with one tacnode and a polynomial for curves with one node and one cusp. Subtracting the last two polynomials from the first, we get $N_{3}(d) = \frac{9}{2}d^{6}-27d^{5}+\frac{9}{2}d^{4} + \frac{423}{2}d^{3}-229d^{2}-\frac{829}{2}d+525.$

We now already note that $N_{r}(d)$ appears to be given by a polynomial in $d$ of degree $2r,$ at least for $r$ small compared to $d$ (conjecturally if $r \leq 2d-2;$ this bound is discussed in Section 3). Further progress mainly came in the 1990s, when polynomiality became clear in greater generality. The intervention of Bell polynomials was first discovered by Kleiman and Piene.
\begin{theorem}
\emph{(Kleiman--Piene, \cite[Theorem 1.1]{KP1}.)} For $r \leq 8$ and $m \geq 3r,$ if $\mathscr{L}$ can be written as $\mathscr{M}^{\otimes m} \otimes \mathscr{N}$ where $\mathscr{M}$ is very ample and $\mathscr{N}$ is globally generated, then $N_{r}(S,\mathscr{L})$ can be written as a polynomial in the four Chern numbers  $\partial = \mathscr{L}^{2}, k=\mathscr{L}\mathscr{K}_{S}, s=\mathscr{K}_{S}^{2},x=c_{2}(S),$ where $\mathscr{K}_{S}$ is the canonical sheaf on $S.$ More specifically, the expressions are
\begin{equation}
N_{r}(S, \mathscr{L}) = P_{r}(\partial,k,s,x)/r!
\end{equation}
where $P_{r}$ is the $r$th complete exponential \textit{Bell polynomial} in certain universal, linear combinations $a_{i}, 1 \leq i \leq r,$ of the four basic Chern numbers mentioned above.
\end{theorem}
As a natural part of their proof, they provided an algorithm whose output was the polynomials $a_{i}.$ They are reproduced below:
\begin{eqnarray*}
a_{1} & = & 3\partial + 2k + x \\
a_{2} & = & -42\partial - 39k - 6s - 7x \\
a_{3} & = & 1380\partial + 1576k + 376s + 138x \\
a_{4} & = & -72360\partial - 95670k - 28842s - 3888x \\
\end{eqnarray*}
\begin{eqnarray*}
a_{5} & = & 5225472 \partial + 7725168 k + 2723400 s + 84384 x \\
a_{6} & = & -481239360 \partial -778065120 k - 308078520 s + 7918560 x\\
a_{7} & = & 53917151040\partial + 93895251840k + 40747613760s - 2465471520x\\
a_{8} & = & -7118400139200\partial - 13206119880240k - 6179605765200s + 516524964480x.
\end{eqnarray*}

Numerous other contributions were made to the understanding of the node polynomials, culminating in the conjectures of G\"ottsche mentioned above, and, perhaps more importantly, his \cite[Conjecture 2.4]{Got}, concerning the generating function of the generic node polynomials $Z_{r}(\partial,k,s,x),$ later proved by Tzeng (using algebraic cobordism theory).
\begin{theorem}
\label{gen_got}
\emph{(Tzeng, \cite[Theorem 1.2]{Tzeng}.)}
Consider the second Eisenstein series $G_{2}(\tau) = -1/24 + \sum_{n=1}^{\infty}\left( \sum_{d|n}d \right) q^{n},$ and the Ramanujan discriminant modular form $\Delta(\tau) = q \prod_{m > 0} (1-q^{m})^{24},$ where $q := e^{2\pi i\tau}$. Let $D$ denote the differential operator $q\frac{d}{dq}.$ There exist universal (independent of $S$ and $\mathscr{L}$) power series $B_{1}$ and $B_{2}$ in $\mathbb{Q}[[q]],$ such that
\begin{equation}
\label{gyz}
\sum_{r \geq 0} Z_{r}(\partial,k,s,x)(DG_{2}(\tau))^{r} = \frac{(DG_{2}(\tau)/q)^{\chi(\mathscr{L})}B_{1}(q)^{\mathscr{K}_{S}^{2}}B_{2}(q)^{\mathscr{L} \mathscr{K}_{S}}}{(\Delta(\tau)D^{2}G_{2}(\tau)/q^{2})^{\chi(\mathscr{O}_{S})/2}}.
\end{equation}
\end{theorem}
This is often referred to as the G\"ottsche--Yau--Zaslow formula. The special cases of $K3$ and Abelian surfaces were proved by Bryan and Leung in \cite{BL1} and \cite{BL2}. Although the power series $B_{1}(q)$ and $B_{2}(q)$ are still not identified, their terms can be computed using the recursive formula for the projective plane of Caporaso--Harris (cf. \cite{CH}), as soon as a threshold value for polynomial validity is known. For instance, G\"ottsche calculated the terms of these power series up to degree 28 (see \cite[Remark 2.5]{Got}), with the conjectural threshold value $r \leq 2d-2.$ We use this formula to calculate the linear polynomials $a_{i}(\partial,k,s,x)$ for $1 \leq i \leq 11.$

It should be mentioned that \textit{tropical geometry} has provided yet another angle of attack for this sort of enumerative problems. Using Mikhalkin's correspondence theorem, Fomin--Mikhalkin gave a combinatorial proof (see \cite[Theorem 5.1]{FM}) of the existence of plane node polynomials $T_{r}(d) \in \mathbb{Q}[d]$ of degree $2r.$ Their threshold value (the maximal value of $r$ for which we have a polynomial expression of $N_{r}(d)$) was $d/2.$ Later (see \cite{Blo}), Block refined the threshold value to $d$ and provided computations of the node polynomials up to $r=14,$ confirming (and extending) the original conjecture of Di Francesco--Itzykson. Nevertheless, these methods have so far not confirmed the general expressions of the cofficients of the node polynomials.\\

\textbf{Structure of this article.} In Sec. 2, we study the node polynomials $Z_{r}(\partial,k,s,x)$ as Bell polynomials in certain universal polynomials $a_{i},$ which we show are integer linear polynomials in $\partial,k,s$ and $x,$ and calculate using the G\"ottsche--Yau--Zaslow formula for all $i \leq 11.$ In Sec. 3, we provide some general remarks on polynomial validity and the shape of the polynomials $a_{i},$ before we proceed to prove Conjecture \ref{plane_shape} of Di Francesco--Itzykson--G\"ottsche in Sec. 4.\\

\textbf{Conventions and notations.} We make no distinction between \textit{line bundles} and \textit{invertible sheaves}. If $\mathscr{L}$ and $\mathscr{K}$ are two line bundles on some variety, we define $\mathscr{LK}$ as $\textnormal{deg }c_{1}(\mathscr{L})c_{1}(\mathscr{K}) \in \mathbb{Z}.$ We consider $c_{2}(S)$ as a number, not as a class, where $S$ is some smooth projective surface. We reserve the notation $T_{r}(d)$ for the node polynomials of $\mathbb{P}^{2}$ equipped with the line bundle $\mathscr{O}_{\mathbb{P}^{2}}(d),$ and use $Z_{r}(\partial,k,s,x)$ to denote the general node polynomials. Thus $T_{r}(d) = Z_{r}(d^{2},-3d,9,3).$ On the other hand, we will simply use $a_{i}(d)$ to denote the $\mathbb{P}^{2}$-polynomial (in $d$) $a_{i}(d^{2},-3d,9,3).$ The generating function $\sum_{r \geq 0} Z_{r}(\partial,k,s,x)q^{r}$ will be denoted by $\phi(S,\mathscr{L})(q).$ Proofs are terminated by a $\Box,$ while the endings of examples and definitions are marked by a $\blacksquare.$\\

\textbf{Acknowledgements.} First and foremost, I am very grateful to my advisor, Professor Ragni Piene, for introducing me to the subject of enumerative geometry, and steadily guiding me towards the present paper. I would also like to thank Professor Steven Kleiman for valuable comments, and the referee for important and constructive remarks. Finally, I am grateful to John Christian Ottem for a very fruitful discussion on the sigma function, and to Karoline Moe for a discussion on functional equations.

\section{Combinatorial approach to node polynomials}
\begin{definition}
\label{bell}
For all integers $r \geq 0,$ the $r$th complete (exponential) Bell polynomial in $r$ variables $a_{i}$ is defined recursively by setting $P_{0} = 1$ and, for all $r \geq 0,$
\begin{equation}
P_{r+1}(a_{1},\ldots,a_{r+1}) = \sum_{k=0}^{r} {r \choose k} P_{r-k}(a_{1}, \ldots, a_{r-k})a_{k+1}.
\end{equation}
This is equivalent to the following formal identity in $t,$
\begin{equation}
\sum_{r \geq 0}P_{r}t^{r}/r! = \exp \left(\sum_{j \geq 1}a_{j}t^{j}/j!\right).
\end{equation}
\end{definition}
\begin{flushright}
$\blacksquare$
\end{flushright}
It is easy to see that these two definitions are equivalent: To see that the formal identity implies the recursion, let $\psi(t)$ denote the generating function $\sum P_{r}t^{r}/r!,$ then differentiate the formal identity
\begin{equation}
\log \psi(t) = \sum a_{j}t^{j}/j!,
\end{equation}
yielding $\psi'(t)/\psi(t) = \sum a_{j}t^{j-1}/(j-1)!,$ thereby getting $\psi'(t) = \psi(t)\sum a_{j}t^{j-1}/(j-1)!.$ Simply by comparing coefficients of the terms in $t^{r},$ we get the recursive relation. The other way around is just as simple.

\begin{definition}
Consider a smooth, irreducible projective surface $S$ over the complex numbers. Let $\mathscr{L}$ be a line bundle on $S.$ We say that $\mathscr{L}$ is \textit{$l$-very ample}, where $l \in \mathbb{N},$ if for all subschemes $Z \subseteq S$ of length $l+1,$ the natural map $H^{0}(S,\mathscr{L}) \rightarrow H^{0}(S,\mathscr{L} \otimes_{\mathscr{O}_{S}} \mathscr{O}_{Z})$ is surjective. Also, recall that the \textit{canonical sheaf} of $S$ is the invertible sheaf $\mathscr{K}_{S} := \bigwedge^{2} \Omega^{1}_{S},$ where $\Omega^{1}_{S}$ is the sheaf of K\"ahler differentials. Finally, if $\mathscr{L}$ is a line bundle on $S,$ the four \textit{Chern numbers} of $(S,\mathscr{L})$ are $\partial := \mathscr{L}^{2}, k := \mathscr{LK}_{S}, s:= \mathscr{K}_{S}^{2}$ and $x := c_{2}(S).$
\end{definition}
\begin{flushright}
$\blacksquare$
\end{flushright}

According to \cite[Theorem 4.1]{KST}, for all $r \geq 0$ there exists a degree $r$ polynomial $Z_{r}$ in four variables, with rational coefficients, such that if $S$ is a smooth projective surface with a line bundle $\mathscr{L}$ which is $r$-very ample, the number of $r$-nodal curves in the linear system $|\mathscr{L}|$ which pass through $\textnormal{dim }|\mathscr{L}|-r$ points in general position on $S$ is given precisely by $Z_{r}(\partial,k,s,x).$ According to \cite[Proposition 2.3]{Got}, there exist four universal power series $A_{i}(q)$ with rational coefficients such that
\begin{equation}
\phi(S,\mathscr{L})(q) := \sum_{r \geq 0} Z_{r}(\partial, k, s, x)q^{r} = A_{1}(q)^{\partial}A_{2}(q)^{k}A_{3}(q)^{s}A_{4}(q)^{x}.
\end{equation}

For given $S$ and $\mathscr{L},$ the generating function $\phi(S,\mathscr{L})(q)$ defined above has integer coefficients, the constant term being 1 (thus its inverse power series has integer coefficients). Write, formally, $(\log A_{i})(q) = \sum_{l \geq 1}b^{(i)}_{l}q^{l},$ where $b^{(i)}_{l} \in \mathbb{Q}.$ Then:
\begin{eqnarray*}
\phi(S,\mathscr{L})(q) & = & \exp \left(\partial \log A_{1}(q) + k \log A_{2}(q) + s \log A_{3}(q) + x \log A_{4}(q)\right) \\
& = & \exp \left( \sum_{l=1}^{\infty} \Bigl(b^{(1)}_{l}\partial + b^{(2)}_{l}k + b^{(3)}_{l}s + b^{(4)}_{l}x\Bigr)q^{l} \right) \\
& = & \exp \left( \sum_{l=1}^{\infty} \frac{a_{l}(\partial,k,s,x)}{l!}q^{l} \right) \\
& = & \sum_{r=0}^{\infty}\frac{P_{r}\Bigl(a_{1}(\partial,k,s,x),\ldots, a_{r}(\partial,k,s,x)\Bigr)}{r!}q^{r},
\end{eqnarray*}
where we have defined $a_{l}(t,u,v,w) := l!\Bigl(b^{(1)}_{l}t + b^{(2)}_{l}u + b^{(3)}_{l}v + b^{(4)}_{l}w\Bigr)$ and used the formal definition of the Bell polynomials $P_{r}.$ Rewrite $a_{l}$ as
\begin{equation}
\label{factorial}
a_{l}(t,u,v,w) = (l-1)!\Bigl(lb^{(1)}_{l}t + lb^{(2)}_{l}u + lb^{(3)}_{l}v + lb^{(4)}_{l}w\Bigr).
\end{equation}
We claim that what is inside the parentheses is a polynomial in the four variables $t,u,v,w$ with integer coefficients. So we must show that $\frac{d}{dq} \log A_{i}(q)$ has integer coefficients. Indeed, for all $S$ and $\mathscr{L},$
\footnotesize
\begin{displaymath}
\frac{d}{dq} \Bigl(\log \phi(S,\mathscr{L})(q)\Bigr) = \Bigl(\phi(S,\mathscr{L})(q)\Bigr)^{-1}\frac{d}{dq} \Bigl(\phi(S,\mathscr{L})(q)\Bigr) =  \left(\partial \frac{d}{dq} \log A_{1}(q) + \ldots + x \frac{d}{dq}  \log A_{4}(q)\right).
\end{displaymath}
\normalsize
Since both $\bigl(\phi(S,\mathscr{L})(q)\bigr)^{-1}$ and $\frac{d}{dq} \bigl(\phi(S,\mathscr{L})(q)\bigr)$ have integer coefficients as power series in $q$, the right-hand side is a power series with integer coefficients which are linear combinations of $\partial,k,s$ and $x.$ Now observe the following lemma.

\begin{lemma}
\label{geom_input}
Suppose $n_{i}, 1 \leq i \leq 4,$ are rational numbers satisfying the following requirements:
\begin{enumerate}[(1)]
\item $n_1$ and $n_4$ are integers;
\item for all polarized, smooth, irreducible projective surfaces $(S,\mathscr{L}),$ the number $n_{1}\partial + n_{2} k + n_{3} s + n_{4} x$ is an integer.
\end{enumerate}
Then each $n_{i} \in \mathbb{Z}.$
\end{lemma}

\begin{proof}
Recall that $\mathbb{P}^{2}$ has Chern numbers $\mathscr{K}_{\mathbb{P}^{2}}^{2} = 9$ and $c_{2}(\mathbb{P}^{2}) = 3.$ Blowing up $\mathbb{P}^{2}$ in a point yields a surface $\mathbb{F}_1$ which has Chern numbers 8 and 4. Let $\mathscr{L}$ be the structure sheaf on $\mathbb{P}^{2},$ then from (2) it follows that $9n_3 + 3n_4 \in \mathbb{Z}.$ Next, consider the structure sheaf on $\mathbb{F}_1,$ so that we get $8n_3 + 4n_4 \in \mathbb{Z}.$ Hence, taking the difference, $n_3-n_4 \in \mathbb{Z}.$ But $n_4$ is an integer, so $n_3$ is as well. Now, let $S$ be the Hirzebruch surface $\mathbb{F}_{3},$ and denote by $E$ the class of the exceptional curve, and $F$ the class of a fiber. Let $\mathscr{L}$ be $\mathscr{O}(1,1) := \mathscr{O}(E+F).$ Then
\begin{eqnarray*}
k & = & \mathscr{K}_{S} \cdot \mathscr{L} \\
& = & (-2E - 5F) \cdot (F + E) = -2EF - 2E^2 - 5F^2 -5FE \\
& = & -2 - 2\cdot (-3) - 0 - 5 = -1,
\end{eqnarray*}
since the canonical divisor is $-2E-(n+2)F$ on $\mathbb{F}_{n}.$ It follows from (1) and (2) that $n_2$ is also an integer.
\end{proof}

So we must prove that $n_1$ and $n_4,$ i.e., the coefficients of respectively $\partial$ and $x$ in the $q^{n}$-term of $\frac{d}{dq}\log \phi(S,\mathscr{L})(q),$ are integers. The fact that $n_1$ is an integer follows direcly from the computation of the $D_{n},$ cf. Eqn. (\ref{eqn:recursion}). Now, a similar procedure, only using a K3 surface $S$ with $\mathscr{L} = \mathscr{O}_{S}$ so that the Chern numbers $\partial, k, s$ are 0 while $x \neq 0,$ allows for an explicit proof that also $n_4,$ the coefficient of $x,$ is an integer, through the G\"ottsche--Yau--Zaslow formula. The process is exactly as in the proof of Prop. \ref{alpha}, only with sligthly different power series. Thus, we see that the coefficients of each $\frac{d}{dq} \log A_{i}(q)$ are integers. Hence, recalling that $P_{r}$ is the $r$th Bell polynomial, the following theorem is a consequence of \cite[Proposition 2.3]{Got}.
\begin{theorem}
\label{cor_bell}
For all $i \geq 1$ there exists a polynomial $a_{i},$ which is $\mathbb{Z}$-linear in four variables, such that for all $r \geq 0,$
\begin{displaymath}
Z_{r}(\partial,k,s,x) = \frac{P_{r}(a_{1}(\partial,k,s,x), \ldots, a_{r}(\partial,k,s,x))}{r!}.
\end{displaymath}
\end{theorem}
The first eight polynomials $a_{i}$ appeared algorithmically in Kleiman--Piene's paper \cite{KP1}. In \cite{Qvi}, we explain their intersection theoretical interpretation. Observing that Noether's formula gives $\chi(\mathscr{O}_{S}) = \frac{1}{12}(c_{1}(S)^{2}+c_{2}(S)) = \frac{s+x}{12},$ while the Hirzebruch--Riemann--Roch formula gives $\chi(\mathscr{L}) = \chi(\mathscr{O}_{S}) + \frac{1}{2}\mathscr{L}(\mathscr{L}-\mathscr{K}_{S}) = \frac{s+x}{12}+\frac{\partial-k}{2},$ the G\"ottsche--Yau--Zaslow formula (\ref{gyz}) can be used to calculate each $a_{i}$ explicitly as a linear combination of $\partial,k,s,x.$ The following algorithm yields the polynomials $a_{i},$ provided that we know the coefficients of the polynomials $B_{i}(q)$ up to the degree we are interested in. These can be calculated using the Caporaso--Harris recursive relation (\cite[Theorem 1.1]{CH}), which, in fact, G\"ottsche did up to degree 28 (\cite[Remark 2.5]{Got}), assuming polynomiality holds for $r \leq 2d-2.$ One could also use Block's algorithm in \cite{Blo}, which is combinatorial in nature and uses labeled floor diagrams from tropical geometry, so that it is independent of the Caporaso--Harris formula. 

\begin{definition}
For each integer $n \geq 1,$ let $\sigma(n)$ denote the sum of the positive divisors of $n,$ including 1 and $n$ itself.
\end{definition}

\begin{algorithm}
Put $y:=DG_{2}(\tau),$ and consider $y$ as a power series in $q:=e^{2\pi i\tau}:$ 
\begin{equation}
y = DG_{2}(\tau) = f(q) = \sum_{n=1}^{\infty}n\sigma(n) q^{n}, \textnormal{ where } \sigma(n) = \sum _{d | n} d.
\end{equation}
Since $f(0) = 0$ while $f'(0) = 1 \neq 0,$ the inversion theorem of Lagrange tells us that we can formally invert the series, to express $q$ as a power series in $y.$ The inverted series has the form

\begin{displaymath}
q = g(y) = \sum_{n=1}^{\infty} \frac{d^{n-1}}{dq^{n-1}} \left(\frac{q}{f(q)}\right)^{n}_{|q=0} \frac{{y}^{n}}{n!}.
\end{displaymath}
Write the G\"ottsche--Yau--Zaslow formula (\ref{gyz}) using the variables $\partial, k, s$ and $x:$
\begin{equation}
\sum_{r = 0}^{\infty} Z_{r}(\partial,k,s,x)y^{r} = \frac{(y/g(y))^{\frac{s+x}{12}+\frac{\partial-k}{2}}B_{1}(g(y))^{s}B_{2}(g(y))^{k}}{\Bigl(\Delta(g(y))D^{2}G_{2}(g(y))/g(y)^{2}\Bigr)^{\frac{s+x}{24}}}.
\end{equation}
Assume we want to calculate $a_{i}$ for $1 \leq i \leq m.$ The algorithm is the following: Start by calculating the terms of the right-hand side up to degree $m.$ Then extract the coefficient of $y;$ this is $a_{1}.$ Assume that we know all $a_{i}$ up to $i=n-1,$ where $n \leq m.$ Then:

\begin{enumerate}
\item extract the coefficient of $y^{n}$ on the right-hand side and multiply it by $n!,$ store this as a variable $P;$
\item write out the $n$th Bell polynomial in $a_{1}, \ldots, a_{n-1}, 0$ (where $a_{1}, \ldots, a_{n-1}$ are replaced by their expressions as functions of $\partial, k, s$ and $x$), store this as a variable $Q;$
\item calculate the difference $P-Q;$ the resulting polynomial in $\partial, k, s, x$ is $a_{n}.$ 
\end{enumerate}
\end{algorithm}

\begin{example}
In the case where $S = \mathbb{P}^{2}$ and $\mathscr{L}$ is the line bundle $\mathscr{O}_{\mathbb{P}^{2}}(d),$ the Chern numbers are $\mathscr{L}^{2} = d^{2}, \mathscr{LK}_{S} = -3d, \mathscr{K}_{S}^{2} = 9$ and $c_{2}(S) = 3.$ Applying the algorithm, we may calculate the polynomials $a_{i}(d) := a_{i}(d^{2},-3d,9,3)$ for $1 \leq i \leq 11.$ They are collected in the table above. The polynomials $\widetilde{a}_{i}$ are obtained by dividing $a_{i}$ by $(-1)^{i-1}(i-1)!.$ \hfill $\blacksquare$
\footnotesize
\begin{center}
\begin{table}
\begin{tabular}{|l|l|}
\hline
$a_{1}(d) =$ & $3 d^2 - 6 d + 3$ \\
\hline
$a_{2}(d) =$ & $-42 d^2 + 117 d - 75$ \\
\hline
$a_{3}(d) =$ & $1380 d^2 - 4728 d + 3798$\\
\hline
$a_{4}(d) =$ & $-72360 d^2 + 287010 d - 271242$\\
\hline
$a_{5}(d) =$ & $5225472 d^2 - 23175504 d + 24763752$\\
\hline
$a_{6}(d) =$ & $-481239360 d^2 + 2334195360 d - 2748951000$\\
\hline
$a_{7}(d) =$ & $53917151040 d^2 - 281685755520 d + 359332109280$\\
\hline
$a_{8}(d) =$ & $-7118400139200 d^2 + 39618359640720 d - 54066876993360$\\
\hline
$a_{9}(d) =$ & $1082298739737600 d^2 - 6363972305913600 d + 9205355815931520$\\
\hline
$a_{10}(d) =$ & $-186244876934645760 d^2 + 1149534771370191360 - 1749876383871100800$\\
\hline
$a_{11}(d) =$ & $35785074342095769600 d^2 - 230648648059354291200 d + 367408987455048537600$\\
\hline
\hline
$\widetilde{a}_{1}(d) =$ & $3 d^2 - 6 d + 3$\\
\hline
$\widetilde{a}_{2}(d) =$ & $42 d^2 - 117 d + 75$\\
\hline
$\widetilde{a}_{3}(d) =$ & $690 d^2 - 2364 d + 1899$\\
\hline
$\widetilde{a}_{4}(d) =$ & $12060 d^2 - 47835 d + 45207$\\
\hline
$\widetilde{a}_{5}(d) =$ & $217728 d^2 - 965646 d + 1031823$\\
\hline
$\widetilde{a}_{6}(d) =$ & $4010328 d^2 - 19451628 d + 22907925$\\
\hline
$\widetilde{a}_{7}(d) =$ & $74884932 d^2 - 391230216 d + 499072374$\\
\hline
$\widetilde{a}_{8}(d) =$ & $1412380980 d^2 - 7860785643 d + 10727554959$\\
\hline
$\widetilde{a}_{9}(d) =$ & $26842726680 d^2 - 157836614730 d + 228307435911$\\
\hline
$\widetilde{a}_{10}(d) =$ & $513240952752 d^2 - 3167809665372 d + 4822190211285$\\
\hline
$\widetilde{a}_{11}(d) =$ & $9861407170992 d^2 - 63560584231524 d + 101248067530602$\\
\hline
\end{tabular}
\caption{The polynomials $a_{i}(d)$ and $\widetilde{a}_{i}(d).$}
\end{table}
\end{center}
\normalsize
\end{example}

\section{Some general remarks on polynomial validity and the $a_{i}$}
When enumerating $r$-nodal curves, there are different approaches depending on how large $r$ is compared to the ampleness of the curve system we consider. When $r$ is sufficiently small, the answers are given by a universal polynomial in $\partial,k,s$ and $x.$ If $r$ becomes too big, we leave the domain of polynomial expressions. This is mirrored in Tzeng's Theorem 1.1 in \cite{Tzeng} by the requirement on $\mathscr{L}$ to be $(5r-1)$-very ample, and in Kool--Shende--Thomas' Theorem 4.1 in \cite{KST}, by the requirement on $\mathscr{L}$ to be $r$-very ample.

So when $S$ and $\mathscr{L}$ are given, there is a specific threshold value $m(S,\mathscr{L})$ which ensures that whenever $r \leq m(S,\mathscr{L}),$ the locus of non-reduced curves and curves with other pathologies has too small dimensions to intervene in the calculations we are doing. In the case of plane curves of degree $d,$ Block shows that for $r \leq d$ we have a polynomial expression for $N_{r}(d)$ (cf. his \cite[Theorem 1.3]{Blo}). We thus have $m(\mathbb{P}^{2},d) \geq d.$ A better threshold, $m(\mathbb{P}^{2},d) = 2d-2,$ is conjectured by G\"ottsche (cf. \cite[Conjecture 4.1]{Got}); indeed, the locus of non-reduced curves of degree $d$ has codimension $2d-1$ in $|\mathscr{O}_{\mathbb{P}^{2}}(d)|,$ so if $r \leq 2d-2,$ the locus of non-reduced curves has smaller dimension than the locus of $r$-nodal curves.

This threshold value is clearly sharp; we cannot have $m(\mathbb{P}^{2},d) > 2d-2.$ For instance, consider the following example (due to Choi, cf. \cite[Remark 3.12]{Choi}): There are $N_{9}(5) = 6930$ 9-nodal curves of degree 5 passing through 11 points in general position in $\mathbb{P}^{2};$ use, for instance, the Caporaso--Harris recursive formula. However, the polynomial $T_{9},$ obtained by whatever method one prefers (for example Block's algorithm in \cite{Blo}; the polynomial figures in his Appendix A), takes the value $-1276110$ in 5, illustrating the sharpness.

As mentioned, Block confirms the validity of the threshold value $r \leq 2d-2$ when $r \leq 14$ \cite[Proposition 1.4]{Blo}. Also, in \cite{KP2}, Kleiman and Piene consider the closure of the set $Y(rA_{1})$ parametrizing $r$-nodal curves of degree $d,$ which is the support of a natural nonnegative cycle $U(r).$ Provided $r \leq 2d-2,$ they show \cite[Theorem 3.1]{KP2} that the degree of the rational equivalence class $[U(r)]$ is equal to the number of $r$-nodal curves of degree $d,$ passing through appropriately many points in $\mathbb{P}^{2}.$ However, they are also restricted by the requirement $r \leq 8$ to ensure polynomiality of this degree. To conclude that $m(\mathbb{P}^{2},d)$ is indeed $2d-2$ with this approach, one would therefore have to extend the polynomiality of deg $[U(r)]$ beyond $r \leq 8.$

Next, we turn to the polynomials $a_{i}(d) = a_{i}(d^{2},-3d,9,3).$ Since $a_{i}$ is linear in its four variables, $a_{i}(d)$ is a quadratic polynomial in $d.$ Also, the Chern numbers $-3d,9$ and 3 are all divisible by 3. Thus the following statement is a consequence of Eq. (\ref{factorial}) and the claim which follows it.
\begin{proposition}
\label{division}
In the case of the complex projective plane, the polynomials $a_{i}(d)$ have the form $(i-1)!(\alpha_{i}d^{2} + \beta_{i}d + \gamma_{i}),$ where $\alpha_{i}, \beta_{i}$ and $\gamma_{i}$ are integers and the latter two are divisible by 3.
\end{proposition}

In fact, we have more than this. We will show that the integer $\alpha_{i}$ is also divisible by 3. This is not a property of the projective plane, but a universal property of the polynomials $a_{i},$ and we will use the G\"ottsche--Yau--Zaslow formula in the case of Abelian surfaces to prove it. 


We also need two small lemmas in our proof.

\begin{lemma}
\label{pow_series}
Let $r \geq 1$ be a positive integer. In the ring of formal power series, $\left(\sum_{n=0}^{\infty}b_{n}q^{n}\right)^{r} = \sum_{n=0}^{\infty}c_{n}q^{n},$ where the coefficients $c_{n}$ satisfy the following recursive formula (provided $b_{0} \neq 0$): $c_{0} = b_{0}^{r},$ and for all $n \geq 1,$
\begin{equation}
c_{n} = \frac{1}{nb_{0}} \sum_{l=1}^{n} (lr-n+l)b_{l}c_{n-l}.
\end{equation}
\end{lemma}

\begin{proof}
See \cite{Hol} for a proof (which deals with the slightly more general case of rational exponents).
\end{proof}

Recall that for each $n \geq 1,$ $\sigma(n)$ denotes the sum of the positive divisors of $n,$ including 1 and itself.
\begin{lemma}
\label{sum_of_divisors}
Let $n \equiv -1[3]$ be a positive integer. Then $\sigma(n)$ is divisible by 3.
\end{lemma}

\begin{proof}
Write $n = \prod_{i=1}^{r}p_{i}^{e_{i}}$ for distinct prime numbers $p_{i}$ and positive integers $e_{i}.$ A well-known identity gives us $\sigma(n) = \prod_{i=1}^{r}(1 + p_{i} + \ldots + p_{i}^{e_{i}}).$ We must show that at least one of these factors is divisible by 3. Since $n \equiv -1[3],$ at least one of the $p_{i}$ must be $\equiv -1[3],$ and it must appear with an odd exponent $e_{i},$ otherwise $p_{i}^{e_{i}} \equiv 1[3].$ Denote it by $p$ and the exponent by $e.$ Then $p + p^{2} + \ldots + p^{e} \equiv -1[3],$ and the factor $1 + p + \ldots + p^{e}$ is divisible by 3, just as we wanted.
\end{proof}

We may prove the divisibility by 3 of $\alpha_{i}$ in $a_{i}(d) = (i-1)!(\alpha_{i}d^2+\beta_{i}d+\gamma_{i}).$

\begin{proposition}
\label{alpha}
For each $i \geq 1,$ the integer $\alpha_{i}$ is divisible by 3.
\end{proposition}

\begin{proof}
For convenience, write $a_{i} = (-1)^{i-1}(i-1)!\left(D_{i}\partial + E_{i}k + F_{i}s + G_{i}x\right)$ where the coefficients are integers. We want to show that $D_{i}$ is divisible by 3. Since the polynomials $a_{i}$ are universal, it suffices to study a particular case, namely the one of Abelian surfaces, where the first and second Chern classes are trivial. So $k = s = x = 0,$ and we are left with $a_{i} = (-1)^{i-1}(i-1)!D_{i}\partial.$ Writing out the G\"ottsche--Yau--Zaslow formula in this case, using $Z_{r}(\partial,k,s,x) = P_{r}(a_{1},\ldots,a_{r}))/r!,$ we get the equality of power series below.
\begin{equation}
\sum_{r=0}^{\infty} \frac{P_{r}(a_{1}, \ldots, a_{r})}{r!} \left( \sum_{n=1}^{\infty} n\sigma(n)q^{n} \right)^{r} = \left( \sum_{n=1}^{\infty} n\sigma(n)q^{n-1}\right)^{\partial/2}.
\end{equation}
Because of the definition of the Bell polynomials by means of the exponential of the generating function of the $a_{i} = (-1)^{i-1}(i-1)!D_{i}\partial$ (see Definition \ref{bell}), the above is equivalent to the equality
\begin{equation}
\sum_{r=1}^{\infty} \frac{2(-1)^{r-1}D_{r}}{r} \left( \sum_{n=1}^{\infty} n\sigma(n)q^{n} \right)^{r} = \log \left( \sum_{n=1}^{\infty} n\sigma(n)q^{n-1}\right).
\end{equation}
Using the power series expansion of the logarithm, we may write this as
\begin{displaymath}
\sum_{r=1}^{\infty} \frac{2(-1)^{r-1}D_{r}}{r}\left(\sum_{n=1}^{\infty}n\sigma(n)q^{n}\right)^{r} = \sum_{r=1}^{\infty}\frac{(-1)^{r-1}}{r} \left(\sum_{n=1}^{\infty} (n+1)\sigma(n+1)q^{n}\right)^{r}.
\end{displaymath}
At this point, it greatly simplifies matters if we differentiate formally, giving:
\begin{eqnarray*}
& & \sum_{r=1}^{\infty}2(-1)^{r-1}D_{r} \left(\sum_{n=1}^{\infty} n\sigma(n)q^{n}\right)^{r-1}\left(\sum_{n=1}^{\infty}n^{2}\sigma(n)q^{n-1}\right) \\
& = & \sum_{r=1}^{\infty}(-1)^{r-1}\left(\sum_{n=1}^{\infty}(n+1)\sigma(n+1)q^{n}\right)^{r-1}\left(\sum_{n=1}^{\infty}n(n+1)\sigma(n+1)q^{n-1}\right).
\end{eqnarray*}
Rewrite this slightly as:
\begin{eqnarray*}
& & \sum_{r=0}^{\infty}2(-1)^{r}D_{r+1} q^{r}\left(\sum_{n=0}^{\infty} (n+1)\sigma(n+1)q^{n}\right)^{r}\left(\sum_{n=0}^{\infty}(n+1)^{2}\sigma(n+1)q^{n}\right) \\
& = & \sum_{r=0}^{\infty}(-1)^{r}q^{r}\left(\sum_{n=0}^{\infty}(n+2)\sigma(n+2)q^{n}\right)^{r}\left(\sum_{n=0}^{\infty}(n+1)(n+2)\sigma(n+2)q^{n}\right).
\end{eqnarray*}

Note that $\sum_{n=0}^{\infty}(n+1)\sigma(n+1)q^{n}$ is a power series with constant term equal to 1, and $\sum_{n=0}^{\infty}(n+2)\sigma(n+2)q^{n}$ is a power series with constant term equal to 6. Now, using Lemma \ref{pow_series}, we may expand the left-hand side above:

\begin{equation}
\label{expand_left}
\sum_{r=0}^{\infty}2(-1)^{r}D_{r+1}q^{r} \left(\sum_{n=0}^{\infty}d_{r}(n)q^{n}\right)\left(\sum_{n=0}^{\infty}(n+1)^{2}\sigma(n+1)q^{n}\right)
\end{equation}
with the recursive formula

\vspace{3mm}

$\left\{\begin{array}{lll}
d_{r}(0) & = & 1 \\
d_{r}(n) & = & \frac{1}{n}\sum_{k=1}^{n} (kr-n+k)(k+1)\sigma(k+1)d_{r}(n-k).
\end{array}\right.$

\vspace{3mm}

By the standard formula for the Cauchy product of two series, (\ref{expand_left}) is a power series of the form
\begin{equation}
\sum_{r=0}^{\infty}2(-1)^{r}D_{r+1}q^{r} \left(\sum_{n=0}^{\infty}v_{r}(n)q^{n}\right),
\end{equation}
where $v_{r}(n) = \sum_{k=0}^{n}(k+1)^{2}\sigma(k+1)d_{r}(n-k).$ In the same way, the right-hand side expands to:
\begin{equation}
\label{expand_right}
\sum_{r=0}^{\infty}(-1)^{r}q^{r} \left(\sum_{n=0}^{\infty}e_{r}(n)q^{n}\right)\left(\sum_{n=0}^{\infty}(m+1)(m+2)\sigma(m+2)q^{m}\right)
\end{equation}
with the recursive formula

\vspace{3mm}

$\left\{\begin{array}{lll}
e_{r}(0) & = & 6^{r} \\
e_{r}(n) & = & \frac{1}{6n}\sum_{k=1}^{n} (kr-n+k)(k+2)\sigma(k+2)e_{r}(n-k).
\end{array}\right.$

\vspace{3mm}

\noindent Expanding the Cauchy product in (\ref{expand_right}), this gives:
\begin{equation}
\sum_{r=0}^{\infty}(-1)^{r}q^{r}\left(\sum_{n=0}^{\infty}w_{r}(n)q^{n}\right),
\end{equation}
where $w_{r}(n) = \sum_{k=0}^{n}(k+1)(k+2)\sigma(k+2)e_{r}(n-k).$ Thus, we have the following equality for all $m \geq 0:$
\begin{displaymath}
\sum_{r+n = m}2(-1)^{r}D_{r+1}v_{r}(n) = \sum_{r+n = m}(-1)^{r}w_{r}(n), \textnormal{ with } r \geq 0, n \geq 0.
\end{displaymath}
This can be written somewhat more simply as
\begin{equation}
\sum_{r = 0}^{m} 2(-1)^{r}D_{r+1}v_{r}(m-r) = \sum_{r = 0}^{m} (-1)^{r}w_{r}(m-r).
\end{equation}
Putting $m=0$ gives $2 \cdot 1 \cdot D_{1} \cdot v_{0}(0) = w_{0}(0),$ hence
\begin{displaymath}
D_{1} = \frac{w_{0}(0)}{2v_{0}(0)} = \frac{2\sigma(2)}{2} = \sigma(2) = 3.
\end{displaymath}
Furthermore, for all $m \geq 1$ the following equality holds:
\begin{displaymath}
2(-1)^{m}D_{m+1}v_{m}(0) = (-1)^{m}w_{m}(0) + \sum_{r=0}^{m-1}(-1)^{r} \Bigl(w_{r}(m-r)-2v_{r}(m-r)D_{r+1}\Bigr).
\end{displaymath}
Using that $v_{m}(0) = 1$ and $w_{m}(0) = 6^{m+1},$ this gives us the following recursive formula: $D_{1} = 3$ and for all $n \geq 2,$
\begin{equation}\label{eqn:recursion}
\boxed{D_{n} = \frac{6^{n}}{2} + \sum_{k=1}^{n-1} (-1)^{n-k}\left(\frac{w_{k-1}(n-k)}{2} - v_{k-1}(n-k)D_{k} \right).}
\end{equation}

We may now proceed recursively to prove that each $D_{i}$ is divisible by 3. It is clearly true for $i=1.$ Assuming validity of the assertion up to $i=n-1$ for some $n \geq 2,$ the above formula shows that it is sufficient to know that each of the integers $w_{k-1}(n-k)$ is divisible by 6 for $1 \leq k \leq n-1.$ Recall the definition of $w_{r}(n) = \sum_{k=0}^{n}(k+1)(k+2)\sigma(k+2)e_{r}(n-k),$ where the $e_{r}(n-k)$ are integers. Thus it is sufficient to show that all integers $M_{n}$ of the form $n(n+1)\sigma(n+1)$ are divisible by 6. If $n+1 \equiv 0[3]$ and is even, it is divisible by 6 and so is $M_{n}.$ If $n+1 \equiv 0[3]$ and is odd, then $n$ is even, and so $n(n+1)$ is divisible by 6 and the same goes for $M_{n}.$ If $n+1 \equiv 1 [3]$ then $n \equiv 0[3]$ and we use the same argument. Using Lemma \ref{sum_of_divisors}, we also get the divisibility when $n+1 \equiv 2[3],$ and we are done.
\end{proof}

\section{Shape conjectures of Di Francesco--Itzykson and G\"ottsche}
In this section, we prove Conjecture \ref{plane_shape} on the shape of the $\mathbb{P}^{2}$-polynomials $T_{r}(d),$ given by $P_{r}(a_{1}(d),\ldots, a_{r}(d))/r!$ from Theorem \ref{cor_bell}. Since $a_{i}(d)$ is a quadratic polynomial in $d,$ while $P_{r}$ is of degree $r$ in $a_{1}, \ldots, a_{r},$ it is clear that the node polynomial $T_{r}(d)$ is a polynomial in $d$ of degree $2r.$ From the observation of the $T_{r}(d)$ for low values of $r,$ Di Francesco and Itzykson originally conjectured (Remark b) following Proposition 2 in \cite{DI}) that $T_{r}(d)$ has the following form:
\begin{equation}
T_{r}(d) = \frac{3^{r}}{r!}\Bigl(q_{0}(r)d^{2r} + q_{1}(r)d^{2r-1} + \ldots\Bigr),
\end{equation}
where the $q_{\mu}$ are polynomials in $r$ of degree $\mu,$ more precisely the following (once the conjecture is stated it is easy to calculate what the polynomials must be).

\begin{eqnarray*}
q_{0}(r) & = & 1  \\
q_{1}(r) & = & -2r \\
q_{2}(r) & = & -\frac{1}{3} r(r-4)\\
q_{3}(r) & = & \frac{1}{6}r(r-1)(20r-13) \\
q_{4}(r) & = & -\frac{1}{54} r(r-1)(69r^{2}-85r+92) \\
q_{5}(r) & = & -\frac{1}{270} r(r-1)(r-2)(702r^{2}-629r - 286) \\
q_{6}(r) & = & \frac{1}{3240} r(r-1)(r-2)(6028r^{3}-15476r^{2} + 11701r + 4425) \\
\end{eqnarray*}
G\"ottsche noted (see \cite[Remark 4.2]{Got}) some further properties which seem to be general for the polynomials $q_{\mu},$ namely that they always take the form
\begin{equation}
q_{\mu}(r) = \frac{1}{\mu!3^{\lfloor \mu/2 \rfloor}} \frac{r!}{(r- \lceil \mu/2 \rceil)!} Q_{\mu}(r),
\end{equation}
$Q_{\mu}(r)$ being a polynomial with integer coefficients and degree $\lfloor \mu/2 \rfloor$ (note that the denominator $\mu!3^{\lfloor \mu/2 \rfloor}$ is not necessarily minimal). This is confirmed by the calculations of Block up to $r=14,$ and in the following theorems, we show that this is general. Let $\mathbb{Q}[X]_{\mu}$ denote the set of polynomials of degree $\mu \geq 0$ with rational coefficients. $m(\mathbb{P}^{2},d)$ denotes the threshold value for polynomiality:

\begin{theorem}
\label{universality}
There exist polynomials $q_{\mu} \in \mathbb{Q}[X]_{\mu}$ such that for all $r \leq m(\mathbb{P}^{2},d),$ the number of $r$-nodal curves of degree $d$ passing through $d(d+3)/2 - r$ points in general position in $\mathbb{P}^{2}$ is given by the polynomial in $d$
\begin{displaymath}
T_{r}(d) = \frac{3^{r}}{r!}\sum_{\mu=0}^{2r} q_{\mu}(r)d^{2r-\mu}.
\end{displaymath}
\end{theorem}

\begin{proof}
We start with the node polynomial equality $T_{r}(d) = P_{r}(a_{1}(d), \ldots, a_{r}(d))/r!$ and expand $P_{r}(a_{1}(d), \ldots, a_{r}(d))$ as $\sum_{\mu=0}^{2r}p_{r,\mu}d^{2r-\mu}.$ We wish to study the coefficients $p_{r,\mu}.$ Define, for each $\mu \in \mathbb{Z},$ the function $p_{\mu}: \mathbb{N} \rightarrow \mathbb{Q}$ by $p_{\mu}(r) = p_{r,\mu}$ if $r \geq \mu/2 \geq0,$ put $p_{\mu}(r) = 0$ in the other cases, in particular when $\mu < 0.$ We show, by induction of $\mu \geq0,$ that there exist polynomials $q_{\mu} \in \mathbb{Q}[X]_{\mu}$ such that we always have $p_{\mu}(r) = 3^{r}q_{\mu}(r)$ when $\mu \geq 0.$ The case $\mu = 0$ is settled, since $\alpha_{1} = 3$ and the leading term of $P_{r}(a_{1}(d), \ldots, a_{r}(d))$ is $\alpha_{1}^{r}d^{2r}.$ Thus we take $q_{0}$ to be the constant polynomial 1.

Recall the recursive relation
\begin{equation}
P_{r+1}(a_{1}, \ldots, a_{r+1}) = \sum_{k=0}^{r} {r \choose k} P_{r-k}(a_{1}, \ldots, a_{r-k})a_{k+1}.
\end{equation}
Expanding both sides as a polynomial in $d$ and identifying terms of degree $2(r+1)-\mu$ in $d,$ we see that the following relation must hold for all $\mu \geq 0$ and all $r+1 \geq \mu/2:$
\footnotesize
\begin{eqnarray*}
p_{\mu}(r+1) & = & \sum_{k=0}^{\lfloor \mu/2 \rfloor} {r \choose k} p_{\mu-2k}(r-k)\alpha_{k+1}k! \\
& + & \sum_{k=0}^{\lfloor (\mu-1)/2 \rfloor} {r \choose k} p_{\mu-2k-1}(r-k)\beta_{k+1}k! + \sum_{k=0}^{\lfloor \mu/2-1 \rfloor} {r \choose k} p_{\mu-2k-2}(r-k)\gamma_{k+1}k!
\end{eqnarray*}
\normalsize
using $a_{i}(d) = (i-1)!(\alpha_{i}d^{2} + \beta_{i}d + \gamma_{i}).$ We use as convention that a sum is 0 if the top index is smaller than the bottom index, and note that here we need the definition that $p_{j}(-)$ is 0 when $j < 0.$ We know that each $\alpha_{i},\beta_{i}$ and $\gamma_{i}$ is divisible by 3 (in particular $\alpha_{1} = 3$). The equality above can be written as

\begin{equation}
 p_{\mu}(r+1)-\alpha_{1}p_{\mu}(r) =  \sum_{k=1}^{\lfloor \mu/2 \rfloor} \frac{r!}{(r-k)!} p_{\mu-2k}(r-k)\alpha_{k+1} + \sum \ldots + \sum \ldots
\end{equation}

Now the right-hand side involves only $p_{j}(-)$ for $0 \leq j < \mu.$ So we may proceed to use an induction argument. Namely, putting $u_{r} := p_{\mu}(r),$ the above is an inhomogeneous linear difference equation of the form $u_{r+1}-3u_{r} = 3^{r} H(r)$ where $H \in \mathbb{Q}[X]$ is some polynomial of degree $\mu-1$ (it depends on the $a_{i}$). Indeed, by the induction hypothesis, we know that for all $s \in \{1, \ldots, \mu\},$ $p_{\mu-s}(r-k) = 3^{r}\frac{q_{\mu-s}(r-k)}{3^{k}},$ where $q_{\mu-s} \in \mathbb{Q}[X]_{\mu}$ is some polynomial of degree $\mu - s$ introduced earlier in the process. Also, $\frac{r!}{(r-k)!}$ is given as $Y_{k}(r)$ where $Y_{k}$ is the degree $k$ polynomial $X(X-1)\ldots (X-k+1).$ Factoring out the exponential $3^{r}$ for all the terms, we are left with a polynomial expression in $r$ with rational coefficients, whose degree is $\mu-1,$ because the only term of this degree appears in the first of the three sums on the right-hand side, for $k=1.$

Now, the solution to such an inhomogeneous equation is the sum of the homogeneous solution ($C \cdot 3^{r}$ for some constant $C$), and a particular solution, which will be a rational polynomial in $r$ of degree $\mu,$ multiplied by $3^{r}.$ This sum is some polynomial, which we define to be $q_{\mu} \in \mathbb{Q}[X]_{\mu},$ multiplied by $3^{r}.$ Note that $q_{\mu}$ must be of degree $\mu,$ because the leading terms of $p_{\mu}(r+1)$ and $\alpha_{1}p_{\mu}(r) = 3p_{\mu}(r)$ cancel each other, and for a lower degree than $\mu,$ we would be left with a polynomial part of degree lower than $\mu-1$ after cancellation, which is not consistent with the right-hand side.
\end{proof}

Now that we know the existence of universal polynomials $q_{\mu}$ such that $T_{r}(d)$ can be written as $\frac{3^{r}}{r!}\sum_{\mu=0}^{2r}q_{\mu}(r)d^{2r-\mu},$ it is natural to ask what these polynomials look like.

\begin{theorem}
\label{general_form}
For all $\mu \geq 0,$ the polynomial $q_{\mu}$ has the form
\begin{equation}
q_{\mu}(X) = \frac{X(X-1)\ldots (X-\lceil \mu/2 \rceil +1)}{\mu!3^{\lfloor \mu/2 \rfloor}}Q_{\mu}(X),
\end{equation}
for polynomials $Q_{\mu}$ of degree $\lfloor \mu/2 \rfloor$ with integer coefficients.
\end{theorem}

\begin{proof}
We have $T_{r}(d) = \frac{P_{r}(a_{1}(d),\ldots,a_{r}(d))}{r!}$ which we may expand as follows:
\begin{displaymath}
T_{r}(d) = \sum_{i_{1} + 2i_{2} + \ldots + ri_{r} = r} a_{1}^{i_{1}}/i_{1}! \cdot (a_{2}/2!)^{i_{2}}/i_{2}! \cdots (a_{r}/r!)^{i_{r}}/i_{r}!.
\end{displaymath}
Replacing $a_{i}(d) = (i-1)!(\alpha_{i}d^{2} + \beta_{i}d + \gamma_{i}),$ we may write
\begin{displaymath}
T_{r}(d) = \sum_{i_{1} + 2i_{2} + \ldots + ri_{r} = r} \prod_{j=1}^{r} \frac{ \Bigl( \left(\alpha_{j}d^{2} + \beta_{j}d + \gamma_{j} \right)/j \Bigr)^{i_{j}}}{i_{j}!}.
\end{displaymath}
Expanding this as a polynomial in $d$ we see that $T_{r}(d)$ is given as
\begin{displaymath}
\sum_{i_{1} + 2i_{2} + \ldots +ri_{r} = r} \frac{1}{1^{i_{1}} \ldots r^{i_{r}} i_{1}! \ldots i_{r}!} \prod_{j=1}^{r} \left(\sum_{u_{j} + v_{j} + w_{j} = i_{j}} \binom{i_{j}}{u_{j}, v_{j}, w_{j}} \alpha_{j}^{u_{j}}\beta_{j}^{v_{j}}\gamma_{j}^{w_{j}} d^{2u_{j} + v_{j}} \right).
\end{displaymath}
For any $0 \leq \mu \leq 2r,$ we identify the coefficient of $d^{2r-\mu}$ in this expression to be:
\begin{displaymath}
\sum_{i_{1} + 2i_{2} + \ldots + ri_{r} = r} \frac{1}{1^{i_{1}} \ldots r^{i_{r}} i_{1}! \ldots i_{r}!} \sum_{\Phi^{(r,\mu)}_{\underline{i}}} \prod_{j=1}^{r} \binom{i_{j}}{u_{j}, v_{j}, w_{j}} \alpha_{j}^{u_{j}}\beta_{j}^{v_{j}}\gamma_{j}^{w_{j}},
\end{displaymath}
where $\Phi^{(r,\mu)}_{\underline{i}}$ is the set of all $(3 \times r)$-tuples $(u_{j},v_{j},w_{j})_{1 \leq j \leq r}$ satisfying $u_{j}+v_{j}+w_{j} = i_{j}$ for each $j,$ and $\sum (2u_{j} + v_{j}) = 2r - \mu.$ This simplifies to
\begin{equation}
\sum_{i_{1} + 2i_{2} + \ldots + ri_{r} = r} \sum_{\Phi^{(r,\mu)}_{\underline{i}}} \prod_{j=1}^{r} \frac{\alpha_{j}^{u_{j}}\beta_{j}^{v_{j}}\gamma_{j}^{w_{j}}}{j^{i_{j}}u_{j}!v_{j}!w_{j}!}.
\end{equation}
The aim now is to provide a clearer presentation of this coefficient. We will see that it can be written as $\frac{3^{r}}{r!}q_{\mu}(r),$ where
\begin{displaymath}
q_{\mu}(r) = \frac{1}{\mu!3^{\lfloor \mu/2 \rfloor}} \frac{r!}{(r- \lceil \mu/2 \rceil)!} Q_{\mu}(r),
\end{displaymath}
$Q_{\mu}(r)$ being a polynomial with integer coefficients and degree $\lfloor \mu/2 \rfloor.$ In the following, we assume $r \geq 2$ (the conjecture is trivial for $r \in \{0,1\}$). Note that for all elements in $\Phi^{(r,\mu)}_{\underline{i}}$ we have
\begin{eqnarray*}
\mu & = & 2r - \sum_{j=1}^{r}(2u_{j}+v_{j}) \\
& = & 2\sum_{j=1}^{r}(j-1)i_{j} + \sum_{j=1}^{r}(v_{j} + 2w_{j}) \\
& \geq & 2i_{2} + 3i_{3} + \ldots + ri_{r} + \sum_{j=1}^{r} (v_{j}+w_{j}).
\end{eqnarray*}
The last inequality holds because $ r \geq 2.$ From this, $\left( \sum_{j=2}^{r} ji_{j} + \sum_{j=1}^{r} (v_{j}+w_{j}) \right)!$ divides $\mu!,$ but clearly $\prod_{j=2}^{r}(ji_{j})!\prod_{j=1}^{r}v_{j}!w_{j}!$ divides the first expression, hence it divides $\mu!$ as well. Since $j^{i_{j}}$ clearly divides $(ji_{j})!$ we get that $\prod_{j=1}^{r}j^{i_{j}}v_{j}!w_{j}!$ divides $\mu!$.

Secondly, we have $r - \lceil \mu/2 \rceil = \lfloor \sum (u_{j} + v_{j}/2) \rfloor \geq \sum u_{j},$ therefore $\prod_{j=1}^{r}u_{j}!$ divides $(r - \lceil \mu/2 \rceil)!.$ We have shown that the numbers $\alpha_{j}, \beta_{j}$ and $\gamma_{j}$ are all divisible by 3 (cf. Lemma \ref{alpha}). Therefore $\prod \alpha_{j}^{u_{j}}\beta_{j}^{v_{j}}\gamma_{j}^{w_{j}}$ is divisible by $3^{\sum i_{j}}$ (recall that $u_{j}+v_{j}+w_{j} = i_{j}$). But $r - \lfloor \mu/2 \rfloor = \lceil r - \mu/2 \rceil = \lceil \sum u_{j} + v_{j}/2 \rceil \leq \sum i_{j},$ so $3^{r-\lfloor \mu/2 \rfloor}$ divides $\prod \alpha_{j}^{u_{j}}\beta_{j}^{v_{j}}\gamma_{j}^{w_{j}}.$ In conclusion, we may write the coefficient of $d^{2r-\mu}$ in $T_{r}(d)$ as
\begin{equation}
\frac{3^{r-\lfloor \mu/2 \rfloor}}{r!} \cdot \frac{r!}{\mu! (r-\lceil \mu/2 \rceil)!} Q_{\mu}(r)
\end{equation}
for some integer $Q_{\mu}(r).$ It remains to show that this is a universal integral polynomial in $r,$ of degree $\lfloor \mu/2 \rfloor.$ The universality is simply a consequence of Theorem \ref{universality}. To show that we have a polynomial with integer coefficients, we start by giving the explicit expression of $Q_{\mu}(r):$
\begin{equation}
Q_{\mu}(r) = \frac{\mu!(r-\lceil \mu/2 \rceil)!}{3^{r-\lfloor \mu/2 \rfloor}} \sum_{\substack{\sum j_{i_{j}}=r\\\Phi^{(r,\mu)}_{\underline{i}}}}\prod_{j=1}^{r} \frac{\alpha_{j}^{u_{j}}\beta_{j}^{v_{j}}\gamma_{j}^{w_{j}}}{j^{i_{j}}u_{j}!v_{j}!w_{j}!}.
\end{equation}
The idea is to consider in more detail the inequalities used above to extract the first factor. Whenever $A$ is an integral multiple of $B,$ we use the notation $A \unrhd B.$ First, we have seen that
\begin{displaymath}
\frac{(r-\lceil \mu/2 \rceil)!}{\prod u_{j}!} \unrhd \frac{(r-\lceil \mu/2 \rceil)!}{(\sum u_{j})!}  =  \frac{(r-\lceil \mu/2 \rceil)!}{(r-\lceil \mu/2 \rceil - \lfloor \sum v_{j}/2 \rfloor)!},
\end{displaymath}
which is a polynomial in $r$ of degree $\lfloor \sum v_{j}/2 \rfloor.$ Secondly, we have
\begin{displaymath}
\frac{\mu!}{\prod j^{i_{j}}v_{j}!w_{j}!} \unrhd \frac{(2r-\sum (2u_{j}+v_{j}))!}{\left(\sum_{j=2}^{r}ji_{j} + \sum_{j=1}^{r}(v_{j}+w_{j})\right)!}.
\end{displaymath}
Here, notice that $(2r-\sum (2u_{j}+v_{j})) - \left(\sum_{j=2}^{r}ji_{j} + \sum_{j=1}^{r}(v_{j}+w_{j})\right)$ is equal to the sum $S := \sum_{j=2}^{r}(j-2)i_{j} + \sum_{j=1}^{r}w_{j},$ hence the right-hand side fraction above is a polynomial in $r$ of degree $S$. In total, we have a polynomial in $r$ of degree
\begin{equation}
\left\lfloor \sum_{j=1}^{r} \left(\frac{v_{j}}{2} + w_{j} \right) \right\rfloor + \sum_{j=2}^{r}(j-2)i_{j} \leq \lfloor \mu/2 \rfloor.
\end{equation}
We get equality in exactly one case, namely when $i_{1} = r.$ Indeed, then all other $i_{j} = 0$ and we have the equalities $u_{1} + v_{1} + w_{1} = r$ and $2u_{1} + v_{1} = 2r-\mu.$ From this we get $2(r-v_{1}-w_{1}) + v_{1} = 2r-\mu,$ which yields $\mu = v_{1} + 2w_{1}.$

In conclusion, the remaining term $Q_{\mu}(r)$ (which we get by summing integral multiples of the above terms over all partitions $\underline{i}$ of $r$ and all $\Phi^{(r,\mu)}_{\underline{i}}$) is an integer polynomial in $r$ of degree $\lfloor \mu/2 \rfloor.$
\end{proof}

\begin{remark}
The part of the proof extracting the coefficient $\frac{3^{r-\lfloor \mu/2\rfloor}}{\mu!(r-\lceil\mu/2\rceil)!}$ was outlined by Kleiman and Piene in \cite{KP3}. In \cite[Theorem 1.1]{Choi}, Choi uses the degeneration method due to Ran (see \cite{Ran}) to show that $T_{r}(d)$ is given by (in his notation)
\begin{displaymath}
T_{r}(d) = \sum_{\mu=0}^{2r}a^{r}_{2r-\mu} \frac{3^{r}d^{2r-\mu}}{r!},
\end{displaymath}
where $a^{r}_{2r-\mu} = r(r-1)\ldots (r-\lfloor (\mu-1)/2 \rfloor)(b^{(\mu)}_{\lfloor \mu/2 \rfloor}r^{\lfloor \mu/2\rfloor} + \ldots + b^{(\mu)}_{1}r+b^{(\mu)}_{0}),$ with $b^{(\mu)}_{i} \in \mathbb{Q}.$ Our Theorem \ref{general_form} is consistent with this result, and also refines it by providing a (multiple of the) common denominator of the coefficients $b^{(\mu)}_{i}.$ It is also consistent with the calculations of $T_{r}(d)$ for $r \leq 14$ done by Block in \cite{Blo} (cf. his Theorem 1.2 and Appendix A).
\end{remark}

\end{document}